\documentclass[preprint,12pt]{elsarticle} 
\usepackage{amsmath} 
\usepackage{amsfonts} 
\usepackage{bm} 
\usepackage{lipsum} 
\usepackage{enumerate} 
\usepackage{graphicx} 
\usepackage{amsfonts} 
\usepackage{mathrsfs} 
\usepackage{epstopdf} 
\usepackage{url} 
\usepackage{verbatim} 
\usepackage{amssymb} 
\usepackage{hyperref} 
\usepackage{color} 
\usepackage{listings} 
\usepackage{float} 
\usepackage{graphicx} 
\usepackage{caption} 
\usepackage{subfigure} 
\usepackage[export]{adjustbox}   
\usepackage{multirow} 
\numberwithin{equation}{section} 
\usepackage[T1]{fontenc} 
\usepackage{natbib} 
\def\@begintheorem#1#2{\par\bgroup{\scshape #1\ #2. }\it\ignorespaces} 
\def\@opargbegintheorem#1#2#3{\par\bgroup% 
	{\scshape #1\ #2\ ({\upshape #3}). }\it\ignorespaces} 
\def\@endtheorem{\egroup}

\if@onethmnum 
\newtheorem{theorem}{Theorem} 
\newtheorem{lemma}[theorem]{Lemma} 
\newtheorem{corollary}[theorem]{Corollary} 
\newtheorem{proposition}[theorem]{Proposition} 
\newtheorem{definition}[theorem]{Definition} 
\newtheorem{example}[theorem]{Example} 
\newtheorem{remark}[theorem]{Remark} 
\newtheorem{homework}[theorem]{Homework} 
\newtheorem{case}[theorem]{} 
 
\else 
\newtheorem{theorem}{Theorem}[section]

\newtheorem{proposition}[theorem]{Proposition} 
\newtheorem{definition}[theorem]{Definition}

\fi 
\begin{document} 
	\bibliographystyle{unsrt} 
\begin{frontmatter} 
\title{Optimal Stirring Strategies for Passive Scalars in a Domain with a General Shape and No-Flux Boundary Condition} 
\cortext[cor1]{Corresponding author} 
 
\author[1]{Sirui Zhu} 
%\ead{zhusr98@zju.edu.cn} 
\author[1]{Zhi Lin} 
%\ead{linzhi80@zju.edu.cn} 
\author[1]{Liang Li} 
%\ead{} 
\author[2]{Lingyun Ding\corref{cor1}} 
\ead{dingly@g.ucla.edu} 
 
\address[1]{School of Mathematical Sciences, Zhejiang University, Hangzhou, Zhejiang 310027, China} 
\address[2]{Department of Mathematics, University of California Los Angeles, Los Angeles, CA, 90095, United States} 
		 
\begin{abstract} 
Multiscale metrics such as negative Sobolev norms are effective for quantifying the degree of mixedness of a passive scalar field advected by an incompressible flow in the absence of diffusion. In this paper we introduce a mix norm that is motivated by Sobolev norm $H^{-1}$ for a general domain with a no-flux boundary. We then derive an explicit expression for the optimal flow that maximizes the instantaneous decay rate of the mix norm under fixed energy and enstrophy constraints. Numerical simulations indicate that the mix norm decays exponentially or faster for various initial conditions and geometries and the rate is closely related to the smallest non-zero eigenvalue of the Laplace operator. These results generalize previous findings restricted for a periodic domain for its analytical and numerical simplicity. Additionally, we observe that periodic boundaries tend to induce a faster decay in mix norm compared to no-flux conditions under the fixed energy constraint, while the comparison is reversed for the fixed enstrophy constraint. In the special case of even initial distributions, two types of boundary conditions yield the same optimal flow and mix norm decay.
  
\end{abstract} 
\end{frontmatter} 
	 
%%%%%%%%%%%%%%%%%%%%%%%%%%%%%%%%%%%%%%%%%%%%%%%%%%%%%%%%%%%%%%%%%%%%%%%% Introduction 
\section{Introduction} 
Mixing of scalars (mass, heat, energy, etc.) by fluid flows is a fundamental phenomenon in a wide variety of natural and industrial fluidic systems \cite{aref2017frontiers,stroock2002chaotic,ding2023shear}. The mixing process can be divided into two distinct stages \cite{Danckwerts1952,Eckart1948,Batchelor_1959}: Flow advection is dominant at first transforming large-scale structures in the initial field into small-scale filaments. As the filament size reduces to the Batchelor scale, molecular diffusion takes over and homogenizes the field by smoothing out any small-scale variances. 
 
It is noteworthy that without advection, achieving uniform homogenization solely through molecular diffusion takes a considerable amount of time even in small regions.  Flow-enhanced mixing has therefore attracted significant research interests for theoretical and practical purposes \cite{stroock2002chaotic,dutta2001dispersion,aminian2016boundaries,ding2021enhanced,ding2022ergodicity}. It is then natural to ask how one should measure the degree of mixing and consequently, what the optimal flow is for various settings.  Spatial variance is a traditional and intuitive candidate to quantify mixing.  However, without diffusion the variance remains constant over time and therefore is unable to characterize mixing dynamics.  By contrast, negative Sobolev norms have gained notable recognition as effective mixing measures in recent years for their guaranteed temporal decay in the ergodic sense even in the absence of diffusion \cite{mathew2005multiscale,mathew2007optimal,thiffeault2012using,Hu2018b,Hu_boundary_2018,HuandZheng2021,Thiffeault_wall_2011, Hassanzadeh2014}. 
 
In particular, \cite{thiffeault2012using, lunasin2012optimal,lin2011optimal} proposed a framework to derive optimal stirring strategies in a multiple-periodic rectangle subject to different flow constraints.  In this setting, the governing advection equation for the scalar field can be readily manipulated into a local-in-time optimization problem for the instantaneous decay rate of the $H^{-1}$ scalar norm.  Consequently, optimal stirring flows were analytically derived and numerically simulated afforded by simplistic Fourier representations.  In more general scenarios, such as impermeable boundaries and non-rectangular domains, the applicability of these methods and results requires further investigation.  

For this reason, in this paper we seek to extend the work in \cite{thiffeault2012using, lunasin2012optimal,lin2011optimal} to more diverse geometric shapes and to no-flux boundary conditions.  While the general framework is inherited here, several theoretical and computational challenges arise and will be addressed, including a proper re-definition of the mix norm, utilizing the Leray-Helmholtz projector to enforce the no-flux condition and to derive explicit optimal stirring flows, as well as robust simulation schemes for general geometries.

The rest of the paper is organized as follows: Section \ref{sec:Problem formulation} surveys various aspects in the formulation of the optimal stirring problem including the governing equation. In Section \ref{sec:Optimal stirring strategy}, we analyze the lower bounds of mixing rates and present the explicit formula for the optimal stirring flow. Section \ref{sec:Numerical simulation} showcases the numerical simulation of the scalar field evolution under the optimal flow for various initial conditions and domains. Additionally, we compare the optimal flow's performance under no-flux and periodic boundary conditions and discuss the impact of different types of boundary conditions on mixing efficiency. Finally, Section \ref{sec:Discussion} concludes the paper by summarizing the main contributions and suggesting potential avenues for future research.

\section{Preliminaries}\label{sec:Problem formulation} 
In this section, we will introduce some key ingredients to the optimization problem we consider, including the governing equation, the cost function as the instantaneous decay rate of a mix norm, the flow constraints and the projection operator necessary for solving the problem with a no-flux boundary.  
 
\subsection{Governing equation for a stirred passive scalar} 
The evolution of a passive scalar field $\theta(\mathbf{x},t)$ stirred by a prescribed flow field $\mathbf{u}(\mathbf{x},t)$ in a bounded domain $\Omega \subset \mathbb{R}^{d}$ is governed by the following advection equation 
\begin{equation}\label{eq:advection equation} 
\begin{aligned} 
\partial_{t}\theta+\mathbf{u}\cdot\nabla\theta=0, \quad\mathbf{x}\in\Omega,\quad  t>0, 
\end{aligned} 
\end{equation} 
with initial condition $\theta(\mathbf{x},0)=\theta_{0} (\mathbf{x})$.   The flow $\mathbf{u}(\mathbf{x},t)$ satisfies the conditions 
\begin{equation}\label{BC} 
\nabla\cdot\mathbf{u}=0, \quad    \left. \mathbf{u} \cdot \mathbf{n} \right|_{\mathbf{x}\in \partial\Omega }=0,  
\end{equation} 
indicating incompressibility and a no-flux boundary where $\mathbf{n}$ denotes the unit outer normal vector of the domain boundary $\partial\Omega$.  Under these settings, the spatial mean of the scalar is conserved at all times and therefore we assume that it vanishes without loss of generality.  That is, 
\begin{equation} 
\overline{\theta}(\mathbf{x},\cdot)=\overline{\theta_{0} }(\mathbf{x})=\frac{1}{|\Omega|} \int_\Omega \theta (\mathbf{x},\cdot) \mathrm{d} \mathbf{x}=0. 
\end{equation}

\subsection{Mix norm  as an equivalent negative Sobolev norm} 
We now introduce a mix norm motivated by negative Sobolev norms used in similar problems with periodic boundaries \cite{mathew2005multiscale, Iyer2013, lin2011optimal}. The Sobolev space $W^{q, p}(\Omega)$ on an open set $\Omega \subset \mathbb{R}^{d}$ for $q\geq 0$ and $1\leq p<\infty$ consists of functions $f$ in $L^{p}(\Omega)$ whose weak partial derivatives $D^{\alpha}f$ up to order $q$ have finite $L^{p}(\Omega)$ norms. The norm on $W^{q, p}(\Omega)$ is given by 
\begin{equation} 
\Vert f\Vert_{W^{q, p}(\Omega)}:=\left(\frac{1}{|\Omega|}\sum_{|\alpha|\leq q} \Vert D^{\alpha}f\Vert_{L^p}^p \right)^{1/p}. 
\end{equation} 
In the case where $p=2$, we denote this Sobolev space as $H^{q}(\Omega):=W^{q, 2}(\Omega)$.  The negative Sobolev space $H^{-q}$ is defined as the dual space of  the positive Sobolev space $H^{q}$. The norm on $H^{-q}$ is defined in terms of the dual space norm as follows: 
\begin{equation}\label{eq:H-q} 
\Vert f\Vert_{H^{-q}}=\Vert f\Vert_{H^{q*}}=\sup_{g\in H^{q}}\frac{\langle f, g\rangle}{\Vert g\Vert_{H^{q}}}, 
\end{equation} 
where $\langle f, g\rangle:=\dfrac{1}{|\Omega|}\displaystyle\int_{\Omega}f(\mathbf{x})g(\mathbf{x})\mathrm{d}\mathbf{x}$ denotes the inner product of two scalar functions.

It is straightforward that without the presence of diffusion, the spatial variance of a mean-zero scalar field, which is the $L^2$ norm, remains constant and would not vanish in time to signify that mixing occurs \cite{thiffeault2012using, Plasting_young_2006,zelati2023mixing}.  Instead of requiring the scalar field to have zero variance to be well-mixed, indicating a strong, pointwise convergence towards a constant, one can define well-mixedness in the weak sense as the following: 
 
\begin{definition} 
A scalar field $\theta$ is well-mixed if  $\dfrac{1}{|\Omega|}\displaystyle\int\limits_{\Omega}^{} f (\mathbf{x}) \theta (\mathbf{x}) \mathrm{d} \mathbf{x}= \overline{\theta}(\mathbf{x}), \forall f \in L^{2} (\Omega) $.    
\end{definition} 
 
Then the negative Sobolev norms are reasonable metrics for mixing justified by the following theorem: 
\begin{theorem}\label{thm:weak convergence negative Sobolev space} 
  Suppose the spatially mean-zero function $\theta$ is bounded uniformly in $L^{2} (\Omega)$, then $\lVert \theta \rVert_{H^{-q}}=0$ ($q>0$) if and only if $\dfrac{1}{|\Omega|}\displaystyle\int\limits_{\Omega}^{} f (\mathbf{x}) \theta (\mathbf{x}) \mathrm{d} \mathbf{x}= 0$ for all $f \in L^{2} (\Omega) $.   
\end{theorem} 
and the fact that even in the absence of diffusion, $\lVert \theta \rVert_{H^{-q}}$ can have a vanishing long-time limit \cite{mathew2005multiscale, lin2011optimal}.

To further facilitate the analytical and numerical investigation conducted in subsequent sections, we define a mix norm equivalent to the $H^{-q}$ norm \eqref{eq:H-q} with $q=1$, in the sense that the two norms induce the same topology on the given domain.  
\begin{definition}[Mix norm]\label{def:mix norm} 
    Let $\Omega \subset \mathbb{R}^d$ be a bounded  domain. For a given scalar function $\theta \in L^2(\Omega)$, the mix norm is defined as 
	\begin{equation}\label{norm} 
	\Vert \theta \Vert_{m}=\Vert\nabla\Delta^{-1}\theta\Vert_{L^2}=\Vert\nabla\varphi\Vert_{L^2}, 
	\end{equation} 
where $\varphi=\Delta^{-1}\theta$ is a solution of the Poisson equation  
\begin{equation}\label{eq:laplace no flux} 
\begin{aligned} 
\Delta\varphi&=\theta \quad \text{in}~~\Omega, \quad  
\left. \nabla\varphi\cdot\mathbf{n} \right|_{\mathbf{x} \in \partial \Omega }=0. 
\end{aligned} 
\end{equation} 
\end{definition} 
The proof of the equivalence of the mix norm and $H^{-1}$ norm are presented in Appendix \ref{sec:Equivalence} and consequently, Theorem \ref{thm:weak convergence negative Sobolev space} also applies to $\Vert \theta \Vert_{m}$. 
 
\subsection{Flow constraints and cost function} 
To formulate an optimization problem, constraints must be applied to the available flow fields. In this paper, we consider two types of constraints: fixed energy constraint and fixed enstrophy constraint. Mathematically, the fixed-energy constrained flow satisfies 
\begin{equation}\label{eq:fixed energy} 
\int_{\Omega}\vert\mathbf{u}(\mathbf{x},t)\vert^2 \mathrm{d}\mathbf{x}=U^2\vert\Omega\vert 
\end{equation} 
for all time, where $U$ represents a given, constant root-mean-square velocity. This constraint is not only natural from a mathematical standpoint but is also relevant in applications where stirring strength is characterized by the kinetic energy density in the flow. 
 
While the fixed energy constraint restricts only the magnitude of the velocity fields, it does not limit the magnitude of the velocity gradient. A reasonable and meaningful constraint that limits the velocity gradient is the fixed enstrophy constraint. Enstrophy is defined as 
\begin{equation} 
\mathcal{E} (\mathbf{u}) =\int_{\Omega} \lVert \nabla \mathbf{u}\rVert_{F}^{2}\mathrm{d}\mathbf{x}=\int_{\Omega} \sum_{i, j=1}^{d}\left(\partial_{x_{j}} u_{i} \right)^{2}\mathrm{d}\mathbf{x}, 
\end{equation} 
where $\mathbf{u}= (u_{1},u_{2},...,u_{d})$ and $\mathbf{x}= (x_{1},x_{2},...,x_{d})$.  The gradient of a vector is defined as the following matrix  $(\nabla \mathbf{u})_{i,j}= \partial_{x_{i}} u_{j}$. $||A ||_{F}$ represents the Frobenius norm of the matrix $A$.  Enstrophy is a measure of the magnitude of vorticity and  is directly related to  dissipation effects in the fluid. In the context of the incompressible Navier-Stokes equations, the enstrophy is the time derivative of the flow energy, $\frac{\mathrm{d}}{\mathrm{d} t} \left( \int_{\Omega}\vert\mathbf{u}(\mathbf{x},t)\vert^2 \mathrm{d}\mathbf{x} \right)=-\nu \mathcal{E} (\mathbf{u})$, 
where $\nu$ is the kinematic viscosity. Due to this property,  the fixed-enstrophy constraint is natural for engineering applications where the focus is on overcoming viscous dissipation to maintain stirring. The constraint is defined as 
\begin{equation}\label{eq:fixed enstrophy} 
\mathcal{E} (\mathbf{u})=\frac{\vert\Omega\vert}{\tau^2}, 
\end{equation} 
where $\tau^{-1}$ represents the rate of strain of the stirring. 
 
Given either of these flow constraints, two types of optimization problems, termed as global- or local-in-time respectively, can be formulated and different optimal stirring strategies ensue \cite{mathew2007optimal,lin2011optimal}.  The distinction arises depending on whether the cost function to be optimized keeps track of the complete evolutionary history of the system.  In this paper, we focus on the local-in-time optimization following the work in \cite{lin2011optimal,  lunasin2012optimal}. In particular, with a given snapshot of the scalar field $\theta (\mathbf{x},t)$,  we seek an incompressible flow subject to the no-flux boundary condition and constraints \eqref{eq:fixed energy} or \eqref{eq:fixed enstrophy}, such that the instantaneous decay rate of the scalar mix norm, $\mathrm{d}\|\theta\|_{m}^{2}/\mathrm{d}t$, is maximized.  As we will see in Section \ref{sec:Optimal stirring strategy},  we can easily rewrite the advection equation (\ref{eq:advection equation}) to represent the chosen cost function as a linear functional of the control stirring $\mathbf{u}$ which is subsequently optimized with constraints.

\subsection{Leray-Helmholtz projection operator} 
Leray-Helmholtz projection operator projects a velocity field to an incompressible velocity field that adheres to no-flux boundary conditions. It is useful in analysis since we focus on the incompressible velocity field. This section summarizes some properties of this operator.  The projection operator is defined as follows 
\begin{equation}\label{eq:Leray-Helmholtz projection} 
\mathscr{P}(\mathbf{v})=\mathbf{v}-\nabla p, 
\end{equation} 
where $p$ is a solution to the Poisson equation subject to a Neumann boundary condition: 
\begin{equation}\label{eq:Leray-Helmholtz projection p} 
\begin{aligned} 
\Delta p&=\nabla\cdot\mathbf{v}\quad\text{in}~~\Omega,\quad \mathbf{n}\cdot\nabla p=\mathbf{n}\cdot\mathbf{v}\quad\text{on}~~\partial\Omega. 
\end{aligned} 
\end{equation} 
This projection  allows us to decompose the flow field $\mathbf{v}$ into a divergence-free part $\mathscr{P} (\mathbf{v})$ that adheres to no-flux boundary conditions and curl free part $\mathbf{v}-\mathscr{P} (\mathbf{v})$. In addition, we have the following proposition.

\begin{proposition}\label{pro:orthogonal divergence free} 
 The curl free part $\mathbf{v}-\mathscr{P} (\mathbf{v})$ in equation \eqref{eq:Leray-Helmholtz projection} is orthogonal to any divergence-free vector field $\mathbf{u}$ subject to no-flux boundary conditions, $\left\langle \mathbf{u},  \mathbf{v}-\mathscr{P} (\mathbf{v})  \right\rangle=0$. 
\end{proposition}       
\begin{proof} 
In fact, we have 
  \begin{equation} 
\begin{aligned} 
|\Omega|\left\langle \mathbf{u},  \mathbf{v}-\mathscr{P} (\mathbf{v})  \right\rangle 
=&\int_{\Omega}\mathbf{u}\cdot\left(\mathbf{v}-\mathscr{P}(\mathbf{v})\right)\mathrm{d}\mathbf{x}=\int_{\Omega}\mathbf{u}\cdot\nabla p\mathrm{d}\mathbf{x}\\ 
=&\int_{\Omega}\nabla\cdot\left(\mathbf{u}p\right)-p\nabla\cdot\mathbf{u}\mathrm{d}\mathbf{x}\\ 
=&\int_{\partial\Omega}p\mathbf{u}\cdot\mathbf{n}\mathrm{d}\mathbf{x}=0, 
\end{aligned} 
\end{equation} 
where $p$ solves the equation \eqref{eq:Leray-Helmholtz projection p}.   The fourth equality applies the divergence theorem and the incompressibility of $\mathbf{u}$. The last equality follows from the no-flux boundary condition of $\mathbf{u}$. This completes the proof.	 
\end{proof} 
 
The following proposition will also be used in deriving the optimal stirring strategy.  
\begin{proposition}\label{pro:orthogonal divergence free 1} 
For all \(\mathbf{v}\in H^1(\Omega)\),  $\left\langle \mathbf{u},  \mathbf{v}- \Delta \mathscr{P} (\Delta^{-1}\mathbf{v})  \right\rangle=0$,  where $\mathbf{u}$ is a divergence-free vector field subject to no-flux boundary conditions. 
\end{proposition}       
\begin{proof} 
  We have 
\begin{equation} 
\begin{aligned} 
&|\Omega|\left\langle \mathbf{u},  \mathbf{v}- \Delta \mathscr{P} (\Delta^{-1}\mathbf{v})  \right\rangle=\int_{\Omega}\mathbf{u}\cdot\left(\mathbf{v}-\Delta\mathscr{P}(\Delta^{-1}\mathbf{v})\right)\mathrm{d}\mathbf{x}\\ 
=&\int_{\Omega}\mathbf{u}\cdot\left(\mathbf{v}-\Delta(\Delta^{-1}\mathbf{v}-\nabla \phi)\right)\mathrm{d}\mathbf{x}=\int_{\Omega}\mathbf{u}\cdot\nabla(\Delta\phi)\mathrm{d}\mathbf{x}\\ 
=&\int_{\partial\Omega}(\Delta\phi)\mathbf{u}\cdot\mathbf{n}\mathrm{d}\mathbf{x}=0,\\ 
\end{aligned} 
\end{equation} 
where $\phi$ solves the equation 
\begin{equation} 
\begin{aligned} 
\Delta\phi&=\nabla\cdot(\Delta^{-1}\mathbf{v})\quad \text{in}~~\Omega,\quad \mathbf{n}\cdot\nabla\phi &=\mathbf{n}\cdot(\Delta^{-1}\mathbf{v})\quad\text{on}~~\partial\Omega.  
\end{aligned}	 
\end{equation} 
 Applying the inverse Laplace operator on vector is defined as $\Delta^{-1} \mathbf{u}= (\Delta^{-1} u_{1},..., \Delta^{-1} u_{d})$. The third step exchanges the gradient and Laplacian operators. The fourth equation involves the divergence theorem and  the incompressibility of $\mathbf{u}$. The last step follows from the no-flux boundary conditions. 
\end{proof}

\section{Optimal stirring strategy}\label{sec:Optimal stirring strategy} 
In this section, we derive the lower bounds of mixing rates and present the explicit formula for the optimal stirring flow.  These analyses apply to general geometries with a no-flux boundary and thereby generalize previous results restricted to a periodic, rectangular domain.
 
\subsection{Lower bounds of the mix norm\label{sec:bounds}} 

We start by multiplying both sides of the equation \eqref{eq:advection equation} by the scalar $\varphi=\Delta^{-1}\theta$ that solves (\ref{eq:laplace no flux}) and integrate it over the domain $\Omega$, 
\begin{equation}\label{eq:advection phi} 
\begin{aligned} 
&\int_{\Omega}\varphi\partial_{t}\theta\mathrm{d} \mathbf{x}+	\int_{\Omega}\varphi \mathbf{u}\cdot\nabla\theta\mathrm{d}\mathbf{x} =0. 
\end{aligned} 
\end{equation} 
Applying the divergence theorem and the no-flux boundary condition reduce the first term on the right hand side to  
\begin{equation} 
\begin{aligned} 
\int_{\Omega}\varphi\partial_{t}\theta\mathrm{d} \mathbf{x}=&\int_{\Omega}\varphi\partial_{t}(\Delta\varphi)\mathrm{d}\mathbf{x} 
=\int_{\Omega}\left(\nabla\cdot(\varphi\partial_{t}\nabla\varphi)-\nabla\varphi\cdot\frac{\partial}{\partial t}(\Delta\varphi)\right)\mathrm{d}\mathbf{x}\\ 
=&\int_{\partial\Omega}\varphi\partial_{t}\nabla\varphi\cdot\mathbf{n}\mathrm{d}\mathbf{x}-\frac{1}{2}\int_{\Omega}\partial_{t}\vert\nabla\varphi\vert^{2}\mathrm{d}\mathbf{x}	=-\frac{1}{2}\frac{d}{dt}\Vert\nabla\varphi\Vert_{L^{2}}^{2}.\\ 
\end{aligned} 
\end{equation} 
The inner product of two vector functions is defined as $\langle \mathbf{u}, \mathbf{v}\rangle:=\frac{1}{|\Omega|}\int_{\Omega}\mathbf{u}(\mathbf{x})\cdot \mathbf{v}(\mathbf{x})\mathrm{d}\mathbf{x}$.  The $L^2$ norm of a vector function is $\lVert \mathbf{u}\rVert_{L^{2}}= \sqrt{\left\langle \mathbf{u}, \mathbf{u} \right\rangle}$.  Similarly, the second term in equation \eqref{eq:advection phi} can be rewritten as 
\begin{equation} 
\begin{aligned} 
\int_{\Omega}\varphi \mathbf{u}\cdot\nabla\theta\mathrm{d}\mathbf{x}&=-\int_{\Omega}\nabla(\varphi\mathbf{u})\cdot\theta\mathrm{d}\mathbf{x}+\int_{\partial\Omega}\theta\varphi \mathbf{u}\cdot\mathbf{n}\mathrm{d}\mathbf{x}\\ 
&=-\int_{\Omega}\theta\mathbf{u}\nabla\varphi\mathrm{d}\mathbf{x}-\int_{\Omega}\theta\varphi\nabla\cdot\mathbf{u}\mathrm{d}\mathbf{x}=-\int_{\Omega}\theta\mathbf{u}\nabla\varphi\mathrm{d}\mathbf{x}\\ 
\end{aligned} 
\end{equation} 
from flow incompressibility.
Combining above results yields 
\begin{equation}\label{eq:norm_dt} 
\frac{\mathrm{d}}{\mathrm{d}t}\|\theta\|_{m}^{2}=\frac{\mathrm{d}}{\mathrm{d}t}\Vert\nabla\varphi\Vert_{L^{2}}^{2}=-2\int_{\Omega}\theta \mathbf{u}\cdot\nabla\varphi\mathrm{d}\mathbf{x},
\end{equation} 
which equates the decay rate of the scalar mix norm (\ref{norm}) and a linear functional of the control stirring $\mathbf{u}$.
 
We first consider the case with the fixed energy constraint. Utilizing the H\"older and Cauchy-Schwarz's inequalities,  we derive the following inequality  
\begin{equation}\label{eq:inequality 1} 
\begin{aligned} 
\frac{\mathrm{d}}{\mathrm{d}t}\|\theta\|_{m}^{2}&=-2\int_{\Omega}\theta \mathbf{u}\cdot\nabla\varphi\mathrm{d}\mathbf{x}\geq-2\Vert\theta\mathbf{u}\Vert_{L^{2}}\Vert\nabla\varphi\Vert_{L^{2}}\\ 
&\geq-2\Vert\theta\Vert_{L^{\infty}}\Vert\mathbf{u}\Vert_{L^{2}}\Vert\nabla\varphi\Vert_{L^{2}}=-2U\vert\Omega\vert^{1/2}\Vert\theta\Vert_{L^{\infty}}\Vert\theta\Vert_{m}, 
\end{aligned} 
\end{equation} 
where the last step follows the fixed energy constraint \eqref{eq:fixed energy}. 
 
Notice that the incompressible advection conserves not only the variance of $\theta$ in time but also the $L^{\infty}$ norm (the supremum of $|\theta|$ over the spatial domain).  Namely, $\lVert \theta (x,t)\rVert_{\infty}= \lVert \theta_{0}\rVert_{\infty}$ at every time $t>0$.  Dividing the inequality \eqref{eq:inequality 1} by $2\Vert\theta\Vert_{m}$ and integrating over time imply: 
\begin{equation}\label{eq:lower bound energy} 
	\Vert\theta(\cdot, t)\Vert_{m}\geq\Vert\theta_{0}\Vert_{m}-U\vert\Omega\vert^{1/2}\Vert\theta_{0}\Vert_{L^{\infty}}t. 
\end{equation} 
This establishes a lower bound for the time at which well-mixedness is achieved: 
\begin{equation} 
t\geq\frac{1}{U\vert\Omega\vert^{1/2}}\frac{\Vert \theta_{0}\Vert_{m}}{\Vert\theta_{0}\Vert_{L^{\infty}}}. 
\end{equation} 
 
Next we consider the flow under fixed enstrophy constraint.  With the assumption $\Vert\nabla\mathbf{u}\Vert_{L^{\infty}}\leq \gamma(t)<\infty$, from the equation \eqref{eq:norm_dt}, we have 
\begin{equation} 
	\frac{\mathrm{d}}{\mathrm{d}t}\|\theta\|_{m}^{2}=2\int_{\Omega}\sum_{i, j=1}^{d} \partial_{x_{j}}\varphi \partial_{x_{i}}\varphi \partial_{x_{j}} u_{i} \mathrm{d}\mathbf{x} 
	\geq-2\gamma(t)\Vert\theta\Vert_{m}^{2}.  
\end{equation} 
An exponential lower bound of the mix norm is obtained by using the Gr\"onwall's  inequality: 
\begin{equation}\label{eq:lower bound enstrophy} 
\Vert\theta(\cdot, t)\Vert_{m}\geq \Vert\theta_{0}\Vert_{m}e^{-\int_{0}^{t}\gamma(s)\mathrm{d}s}. 
 \end{equation} 
 
It should be noted that the lower bounds \eqref{eq:lower bound energy} and \eqref{eq:lower bound enstrophy} are applicable to any incompressible velocity field. Consequently, there may be a gap between these formulas and the mix norm of the scalar field advected by the local-in-time optimal flow. 
       
\subsection{Optimal stirring strategy with fixed energy constraint}
We now investigate the optimal mixing flow field under the constraint of fixed energy.   The decay rate \eqref{eq:norm_dt} to be optimized can be rewritten using the projection operator as: 
\begin{equation}\label{eq:norm dt fixed energy} 
\begin{aligned} 
&\frac{\mathrm{d}}{\mathrm{d}t}\|\theta\|_{m}^{2}=-2\int_{\Omega}\mathbf{u}\cdot\left(\mathscr{P}(\theta\nabla\varphi)+\theta\nabla\varphi-\mathscr{P}(\theta\nabla\varphi)\right)\mathrm{d}\mathbf{x}=-2\int_{\Omega}\mathbf{u}\cdot\mathscr{P}(\theta\nabla\varphi)\mathrm{d}\mathbf{x}, 
\end{aligned} 
\end{equation} 
where the second step follows from the proposition \ref{pro:orthogonal divergence free}, $\left\langle \mathbf{u}, \theta\nabla\varphi-\mathscr{P}(\theta\nabla\varphi) \right\rangle=0$.

To minimize the decay rate functional \eqref{eq:norm dt fixed energy} we introduce a Lagrange multiplier to impose the fixed energy constraint to the flow and the cost functional is defined as:
\begin{equation}\label{eq:Lagrange multiplier formalism} 
J_{e}(\mathbf{u})=-2\int_{\Omega}\mathbf{u}\cdot \mathscr{P}(\theta\nabla\varphi)\mathrm{d}\mathbf{x}+\lambda\left(\int_{\Omega}|\mathbf{u}|^{2}\mathrm{d}\mathbf{x}-U^{2}|\Omega|\right). 
\end{equation} 
and its the G\^ateaux derivative is
\begin{equation} 
dJ_{e} (\mathbf{u};\mathbf{v}):=\lim\limits_{\epsilon\to 0}\frac{J_{e}(\mathbf{u}+\epsilon\mathbf{v})-J_{e}(\mathbf{u})}{\epsilon}=2\int_{\Omega}\mathbf{v}  \cdot\left(-\mathscr{P}(\theta\nabla\varphi)+\lambda \mathbf{u} \right)\mathrm{d}\mathbf{x}, 
\end{equation} 
where both $\mathbf{v}$ and $\mathbf{u}$ satisfy the no-flux boundary condition.  Since the optimal control $\mathbf{u}_e$ solves $dJ_{e} (\mathbf{u}_e;\mathbf{v})=0$ for any $\mathbf{v}$, we have $\lambda \mathbf{u}_e=\mathscr{P}(\theta\nabla\varphi)$.  The energy constraint further specifies the value of $\lambda$ and the optimal stirring is therefore 
\begin{equation}\label{eq:optimal energy} 
\mathbf{u}_e=U\frac{\mathscr{P}(\theta\nabla\varphi)}{\langle|\mathscr{P}(\theta\nabla\varphi)|^{2}\rangle^{1/2}}. 
\end{equation} 
The strategy \eqref{eq:optimal energy} resembles the optimal stirring derived in \cite{lunasin2012optimal} and \cite{lin2011optimal}. However, it is a significant extension for its broader applicability for domains with more general shape and no-flux boundary conditions. 
 
\subsection{Optimal stirring strategy with fixed enstrophy constraint\label{sec:optmix_fixed_ens}} 
This subsection is dedicated to the optimal mixing under the fixed enstrophy constraint and we will see how non-periodic boundaries poses a new challenge.  

Similar to the fixed energy case, we first rewrite mix norm decay rate  \eqref{eq:norm_dt} as follows: \begin{equation} 
\begin{aligned} 
&\frac{\mathrm{d}}{\mathrm{d}t}\|\theta\|_{m}^{2}=-2\int_{\Omega}\mathbf{u}\cdot\left(\Delta\mathscr{P}(\Delta^{-1}(\theta\nabla\phi))+\theta\nabla\phi-\Delta\mathscr{P}(\Delta^{-1}(\theta\nabla\phi))\right)\mathrm{d}\mathbf{x}\\ 
&=-2\int_{\Omega}\mathbf{u}\cdot\Delta\mathscr{P}(\Delta^{-1}(\theta\nabla\phi))\mathrm{d}\mathbf{x}.
\end{aligned} 
\end{equation} 
 The second step follows from proposition \ref{pro:orthogonal divergence free 1}.  
 
To incorporate the fixed enstrophy constraint \eqref{eq:fixed enstrophy}, we defined the cost functional as: 
\begin{equation} 
J_{en}(\mathbf{u})=-2\int_{\Omega}\mathbf{u}\cdot\Delta\mathscr{P}\Delta^{-1}(\theta\nabla\varphi)\mathrm{d}\mathbf{x}+\lambda\left(\int_{\Omega}\sum_{i, j=1}^{d}\left(\partial_{x_{j}} u_{i}\right)^{2}\mathrm{d}\mathbf{x}-\frac{|\Omega|}{\tau^{2}}\right). 
\end{equation} 
Computing the functional derivative yields: 
\begin{equation}\label{eq:dJpuv} 
\begin{aligned} 
dJ_{en} (\mathbf{u};\mathbf{v})&:=  \lim\limits_{\epsilon\to 0}\frac{J_{p}(\mathbf{u}+\epsilon\mathbf{v})-J_{p}(\mathbf{u})}{\epsilon}=-2\int_{\Omega}\mathbf{v}\cdot\Delta\mathscr{P}\Delta^{-1}(\theta\nabla\varphi)\mathrm{d}\mathbf{x}+2\lambda\int_{\Omega}\sum_{i, j=1}^{d}\partial_{x_{j}} u_{i} \partial_{x_{j}} v_{i} \mathrm{d}\mathbf{x}\\ 
  &=-2\int_{\Omega}\mathbf{v}\cdot \left( \Delta\mathscr{P}(\Delta^{-1}(\theta\nabla\varphi))+\lambda \Delta \mathbf{u}  \right)\mathrm{d}\mathbf{x}+2\lambda \int_{\partial\Omega} \sum\limits_{i,j=1}^{d}v_{i}n_{j}\partial_{x_{j}}u_{i}\mathrm{d}\mathbf{x}, 
\end{aligned} 
\end{equation} 
where both $\mathbf{v}$ and $\mathbf{u}$ satisfy the no-flux boundary condition, $\mathbf{n}= (n_{1},n_{2},...,n_{d})$. 
 
If we further assume that the boundary integral, $\int_{\partial\Omega} \sum\limits_{i,j=1}^{d}v_{i}n_{j}\partial_{x_{j}}u_{i}\mathrm{d}\mathbf{x}$, in the above computation vanishes, which is guaranteed in the periodic case,  the requirement $dJ_{en} (\mathbf{u};\mathbf{v})=0$ for all $\mathbf{v}$ implies 
\begin{equation}
	\Delta\mathscr{P}(\Delta^{-1}(\theta\nabla\varphi))+\lambda \Delta \mathbf{u}=0.
\end{equation}
  Then we retrieve the optimal stirring field with given enstrophy as  
\begin{equation}\label{eq:optimal enstrophy} 
\mathbf{u}_{en}=\frac{1}{\tau}\frac{-\mathscr{P}(\Delta^{-1}(\theta\nabla\varphi))}{\langle|\nabla\mathscr{P}(\Delta^{-1}(\theta\nabla\varphi))|^2\rangle^{1/2}}. 
\end{equation} 

In general, the boundary integral in (\ref{eq:dJpuv}) is non-zero for non-periodic boundaries. However, we consider the expression (\ref{eq:optimal enstrophy}) valuable for three reasons. First, in a rectangular domain the boundary integral is indeed zero which can be shown by making an even extension for the quantities in the original rectangle with a no-flux boundary, to a larger rectangle with a periodic boundary. We require the evenness of the extension in order to maintain the no-flux condition at all times. Then the symmetry ensures that the integral along opposite edges cancels out resulting an overall vanishing contribution. Consequently the flow \eqref{eq:optimal enstrophy} is the optimal strategy in this context and further details are provided in Appendix \ref{sec:Even extension for the square domain}.  Second, numerical simulations in the subsequent section demonstrate that the  strategy \eqref{eq:optimal enstrophy} is still a very efficient mixer in the sense that it also  produces an exponential decay in the mix norm even when the domain is not rectangular.  Last but not the least, the relatively simple expression serves as a convenient subject for further analysis and computation in scalar mixing problems.  For these considerations, we refer to the flow \eqref{eq:optimal enstrophy} as the ``quasi-optimal'' stirring.

\section{Numerical simulation}\label{sec:Numerical simulation} 
In this section we simulate the evolution of the passive scalar field advected by the optimal strategies \eqref{eq:optimal energy} and \eqref{eq:optimal enstrophy} respectively.  While the results apply to domains in any finite-dimensional space $\mathbb{R}^{d}$, we focus on the two-dimensional cases where $d=2$ hereafter for the sake of clarity.  The scalar governed by \eqref{eq:advection equation} and the associated optimal flows are computed via the spectral method implemented with Matlab in 2D rectangles. For general domains such as circles, annuli and etc., we employ the finite element algorithm, which is implemented via  Freefem++, a versatile finite element method (FEM) software package \cite{hecht2012new}. The details of these algorithms can be found in Appendix \ref{sec:numerical_methods}. 
 
\subsection{Energy constrained simulations with no-flux boundary conditions} 
In the following simulations, we set the initial condition to be a superposition of two one-dimensional sinusoidal waves as
\begin{equation}\label{eq:energy_theta} 
\theta_{0}(x,y)=0.5\sin(\pi x)+0.25\sin(2\pi y), 
\end{equation} 
and the parameter in the fixed energy constraint \eqref{eq:fixed energy} to $U=1$.
 
\subsubsection{Optimal stirring in a square under no-flux boundary conditions} 
 
\begin{figure}[ht] 
	\centering 
	\includegraphics[width=1\linewidth]{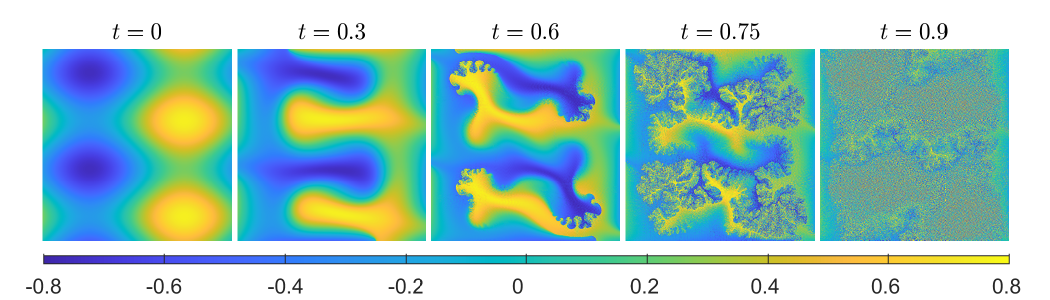} 
	\caption{Evolution of the scalar field in $[-1, 1]^2$ with fixed energy constraint \eqref{eq:fixed energy} and $U=1$ with initial condition \eqref{eq:energy_theta} that has higher scalar concentration on the right half of the square.} 
\label{fig:energy_eg_square}	 
\end{figure} 
\begin{figure}[ht] 
  \centering 
  \subfigure[]{ 
\includegraphics[width=0.46\linewidth]{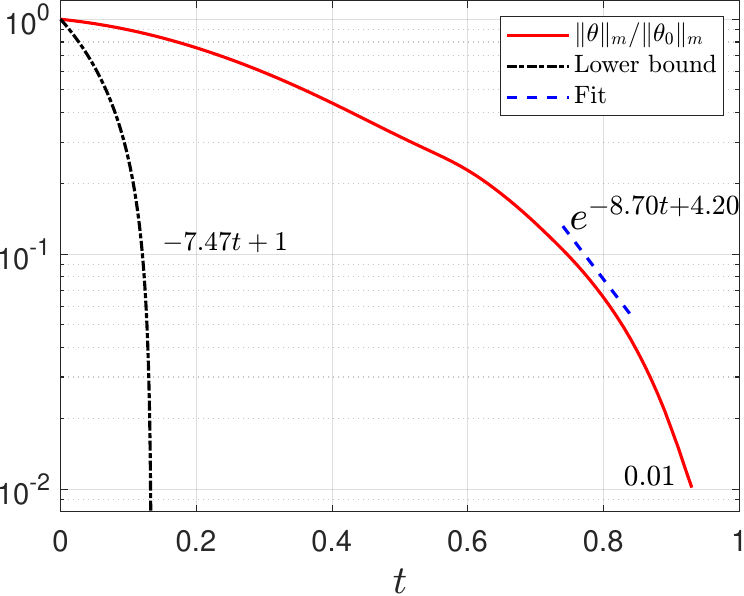} 
} 
\subfigure[]{ 
  \includegraphics[width=0.46\linewidth]{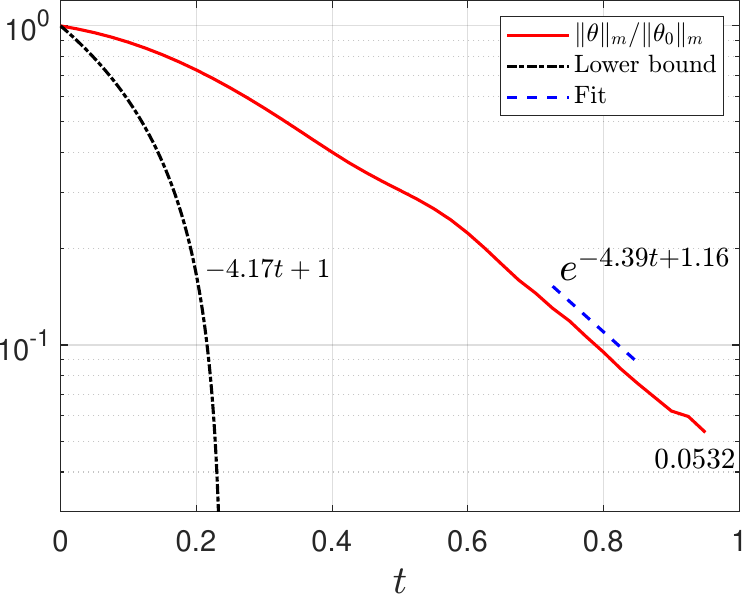} 
  } 
  \hfill 
  \caption[] { Semi-logarithmic plot of mix norms under fixed energy constraint for (a). the square $[-1,1]^2$ and for (b). the circular region (\ref{eq:circular domain}) with the same initial condition \eqref{eq:energy_theta}. The red, solid curves represent the renormalized mix norm. The blue, dashed lines represent the fitted exponential function of the mix norm. The black, dot-dashed curves represent the theoretical lower bound (\ref{eq:lower bound energy}).} 
  \label{fig:energy mix norm} 
\end{figure} 

For the simple square:
\begin{equation}\label{eq:square domain} 
\begin{aligned} 
\left\{ (x,y) | x,y\in [-1, 1] \right\}, 
\end{aligned} 
\end{equation} 
we use a spatial discretization of $1025^2\approx 10^6$ points and a time step size of $\Delta t=0.01$.  Figure \ref{fig:energy_eg_square} illustrates the evolution of the scalar field with the initial condition \eqref{eq:energy_theta}. Initially, advection disrupts the large-scale features of the scalar field and transports the scalar from right to the left. Meanwhile the no-flux condition prohibits scalar passage through the boundary and tangentially-stretched, less-mixed layers are formed most notably at left and right edges.  Eventually stirring demolishes large-scale variances and cascades to smaller and smaller scales due to the interplay between the scalar and the optimal strategy established in \eqref{eq:fixed energy}.  As a result, more intricate fractal-like structures develop as time progresses. 

There is an additional feature for the optimal stirring under the fixed energy constraint when the flow is unlimited to re-distribute towards stronger velocity gradients associated with efficient small-scale mixing.  Correspondingly, the bulk, large-scale transport is quickly suppressed and is negligible after $t=0.6$.  This also suggests that this constraint might lead to an optimal stirrer producing an even faster mix norm decay compared to fixed enstrophy.  This has also been hinted in Section \ref{sec:bounds} which suggested a possible well-mixedness at a finite time for fixed energy, whereas the mix norm decays at most exponentially and vanishes only in the asymptotic limit for fixed enstrophy.  

Panel (a) in Figure \ref{fig:energy mix norm} supports this conjecture in which the normalized mix norm is plotted in time with the red, solid curve along with the bound (\ref{eq:lower bound energy}) represented by the black, dot-dashed curve.   As time advances, the norm decreases super-exponentially and reaches approximately 1\% of its initial value at $t=0.9$.   An exponential decay rate is fitted to be $8.70$ for $t\in[0.75, 0.85]$.  In this setting, the bound \eqref{eq:lower bound energy} reduces to $(1-7.47t)\Vert\theta_{0}\Vert_{m}$ which for all time remains below the evolving mix norm as predicted. However, the substantial gap between the lower bound and the mix norm optimized locally in time suggests that there is room for improving the lower bound or further enhancing the overall decay with a global-in-time control.

\subsubsection{Optimal stirring in a circle under no-flux boudary conditions} 
The domain considered here is a circle of radius $R=2/\sqrt{\pi}$ centered at the origin: 
\begin{equation}\label{eq:circular domain} 
\left\{ (x,y)| x^{2}+y^{2}\leq \frac{4}{\pi} \right\},
\end{equation} 
whose area is the same as the square $[-1,1]^2$ we considered above. In our simulations, the finite-element mesh consists of 196675 vertices and 391848 triangles, the minimal edge size of the triangulation is $0.00334173$, the maximal is $0.00948779$. The time step size is $\Delta t=0.025$. 
 
\begin{figure}[ht] 
\centering 
\includegraphics[width=0.7\textheight]{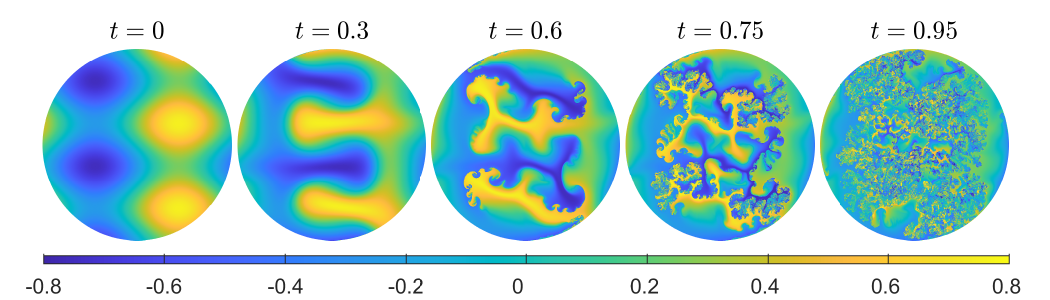} 
   \caption{Evolution of the scalar field with fixed energy constraint  \eqref{eq:fixed energy} and $U=1$ in a circle of radius $2/\sqrt{\pi}$.  The initial condition is provided in equation  \eqref{eq:energy_theta}.} 
\label{fig:energy_eg_circle}	 
    \end{figure} 
     The scalar plots in Figures \ref{fig:energy_eg_square} and \ref{fig:energy_eg_circle} reveal similar mixing patterns.  Panel (b) of Figure \ref{fig:energy mix norm} illustrates the evolution of the normalized mix norm. In the circular case, the mix norm exhibits a stable,  exponential decay with a fitted rate of $4.39$.  It should be noted that this rate deviates significantly from the square case suggesting a strong correlation between the decay rate and the specific geometry.  In this setting, it can be computed that  $\Vert\theta_{0}\Vert_{m}=0.3596, \Vert\theta_{0}\Vert_{L^{\infty}}=0.75$ and the lower bound \eqref{eq:lower bound energy} reduces to $(1-4.17t)\Vert\theta_{0}\Vert_{m}$.

\subsection{Enstrophy constrained simulations with no-flux boundary conditions} 
We now turn to optimal stirring with a fixed enstrophy constraint and will demonstrate the general applicability of our results in various geometries.

\subsubsection{Optimal stirring in a square under no-flux boundary conditions}  
 
\begin{figure} 
	\centering 
	\includegraphics[width=1\linewidth]{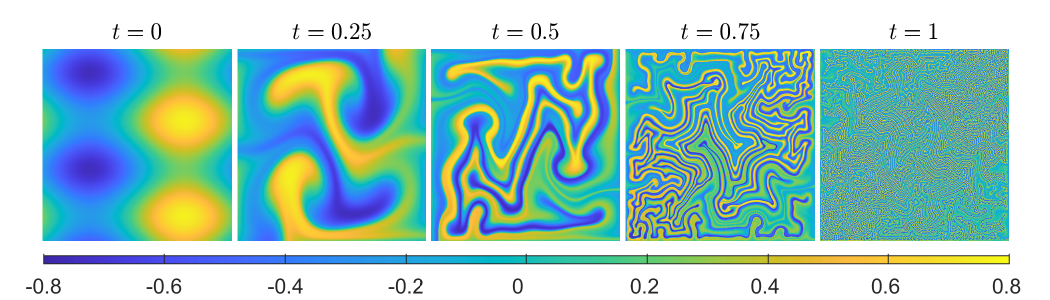} 
	\caption{Evolution of the scalar field in $[-1, 1]^2$ with fixed enstrophy constraint \eqref{eq:fixed enstrophy} and $\tau^{-1}=15$, the initial condition is \eqref{eq:energy_theta}.} 
	\label{fig:enstrophy square}	 
\end{figure} 
Figure \ref{fig:enstrophy square} illustrates the evolution of the scalar field in $[-1, 1]^2$ uniformly discretized with $513^2$ points with the initial condition \eqref{eq:energy_theta}.  Here the parameter in the fixed enstrophy constraint \eqref{eq:fixed enstrophy} is set as $\tau^{-1}=15$ and the time step size is  $\Delta t=0.01$. The mixing pattern is visibly different from those shown in Figures \ref{fig:energy_eg_square} and \ref{fig:energy_eg_circle}.  But the general cascading towards smaller scales in both the velocity and the scalar fields are maintained, as well as the overall exponential decay rate.
   
In contrast to the optimal stirrer with fixed energy, the fixed enstrophy puts a limits onthe velocity gradient.  Therefore a large-scale flow persists and can be easily observed accompanying the emergence and growth of finer structures. This mechanism  eventually leads to a more uniform scalar field in the large-scale sense without apparent boundary layers as seen in the fixed energy case.

\begin{figure} 
	\centering 
	\includegraphics[width=0.45\linewidth]{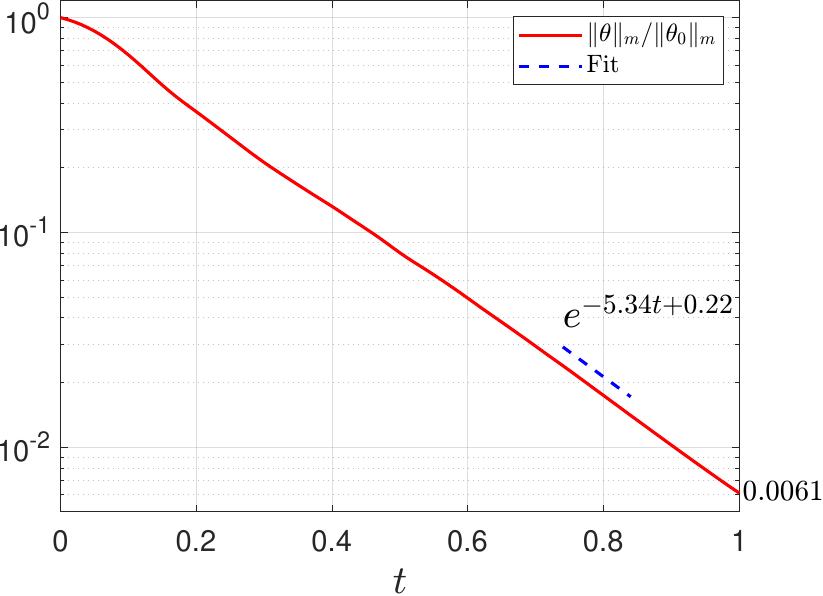}	 
	\caption{Semi-logarithmic plot of mix norm in the square with the initial scalar \eqref{eq:energy_theta} under fixed enstrophy constraint. The red solid line represents the mix norm. The dashed blue line represents the fitted exponential function of the mix norm.} 
	\label{fig:enstrophy square norm} 
\end{figure} 
 
\subsubsection{Quasi-optimal stirring in general domains under no-flux boundary conditions} 
In this section, we will extend the simulation to domains of general, non-rectangular shapes for which we have established in Section \ref{sec:optmix_fixed_ens} that the stirring strategy (\ref{eq:optimal enstrophy}) is only quasi-optimal.  

Here we consider three different geometries: 1. The circle \eqref{eq:circular domain}; 2. The non-convex, L-shaped region 
\begin{equation}\label{eq:L domain} 
\left\{ (x,y)| x\in [0,\sqrt{5}], y\in[0,\sqrt{5}-1] \right\}\cup\left\{ (x,y)| x\in [0,\sqrt{5}-1], y\in[\sqrt{5}-1,\sqrt{5}] \right\}; 
\end{equation} 
And 3. The annulus 
\begin{equation}\label{eq:donut domain} 
\left\{ (x,y)| \frac{1}{\pi}\leq x^{2}+y^{2}\leq \frac{5}{\pi}\right\},
\end{equation} 
which is not simply-connected.

 The area of all three domains is $4$. The parameter in the fixed enstrophy constraint \eqref{eq:fixed enstrophy} is set universally as $\tau^{-1}=15$ and the numerical time step size is $\Delta t=0.01$.

\begin{figure} 
\centering 
\subfigure[]{ 
    \includegraphics[width=1\linewidth]{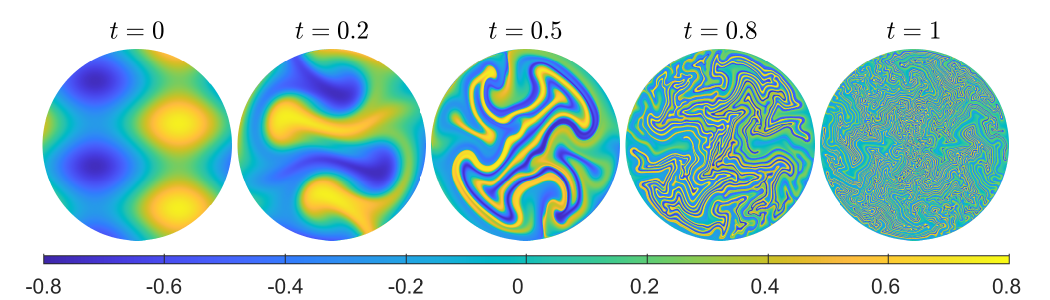} 
  } 
\subfigure[]{ 
	\includegraphics[width=1\linewidth]{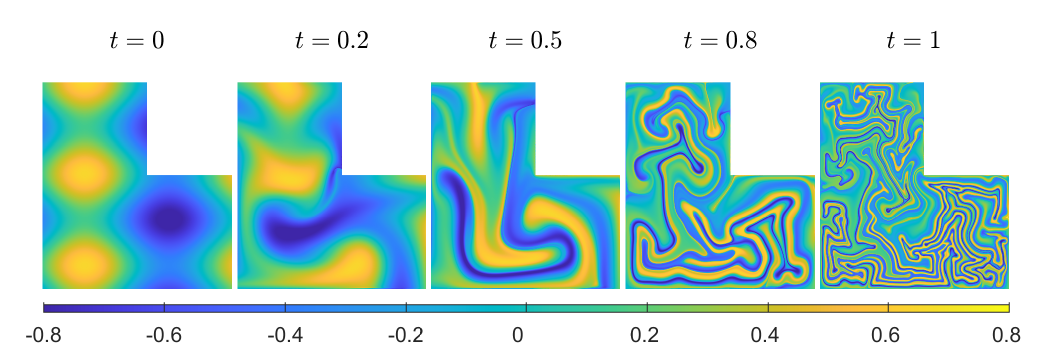} 
     } 
\subfigure[]{ 
      \includegraphics[width=1\linewidth]{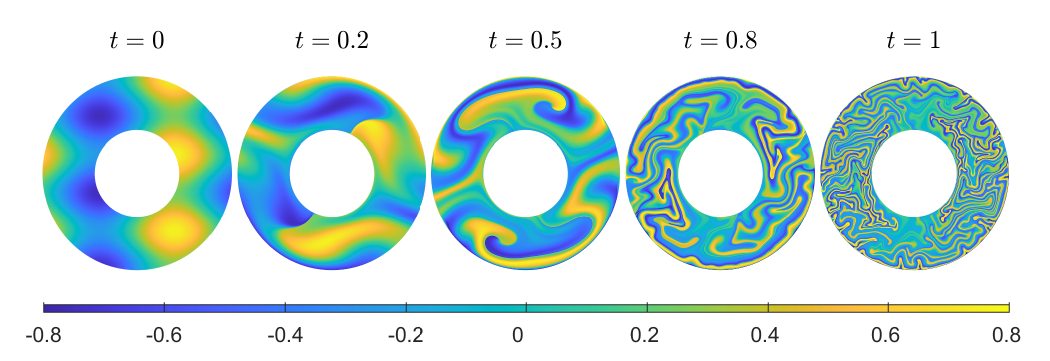} 
      } 
\caption{Evolution of the scalar field with the initial condition \eqref{eq:energy_theta}  under fixed enstrophy constraint in  (a). a circular region, in (b). a L-shaped region and in (c). an annulus. The constraint parameter is $\tau^{-1}=15$.} 
	\label{fig:enstrophy circle}	 
\end{figure} 
 
Figure \ref{fig:enstrophy circle} showcases the evolution of the scalar field mixed by the proposed quasi-optimal stirring strategy with the initial condition given by \eqref{eq:energy_theta} across these three domains.  Although qualitative differences are obvious, a reasonably degree of well-mixedness is reached in every geometry with the most homogenized scalar being found in the regular circle.  In the L-shape domain and annulus, the mixing degree is comparable but less homogenized compared to the circular domain. Further justification of the quasi-optimal strategy is provided in Figure \ref{fig:enstrophy norms} where all mix norms exhibit exponential decay.  Again, the differences in specific rates point to the impact of detailed geometry on the mixing process.   

This motivates our study of the eigenvalues among all descriptors for  the Laplace operator \cite{kuttler1984eigenvalues}. Since the mix norm is defined with the inverse of the Laplace operator, it is only reasonable to hypothesize that the eigenvalue of the Laplace operator could play a big role in determining the decay rate of the mix norm.  The eigenfunctions $\phi_{k}$ and their corresponding eigenvalues $\lambda_{k}$ for the Laplace operator satisfy the following equation: 
 
\begin{equation} 
\begin{aligned} 
\Delta\phi_{k}&=-\lambda_{k} \phi_{k} \quad \text{in}~~\Omega, \quad \left. \nabla\phi_k\cdot\mathbf{n} \right|_{\mathbf{x} \in \partial \Omega }=0, 
\end{aligned} 
\end{equation} 
 
In the case of a bounded domain, these eigenvalues form an increasing sequence: $0=\lambda_{0}<\lambda_{1}\leq \lambda_{2}\leq...\leq \lambda_{k}\leq...$. The smallest non-zero eigenvalue, denoted as $\lambda_{1}$, for different domains is presented in Table \ref{tab:eigenvalue}. A noteworthy observation is that the eigenvalue of the Laplace operator for the square $[-1,-1]^2$ is similar to that for the circular (\ref{eq:circular domain}), and correspondingly the normalized mix norms at the end of of the simulations when $t=1$ are approximately the same for these two geometries. On the other hand, $\lambda_{1}'s$ for the L-shaped region and for the annulus are both nearly half of those for the square and circle.  As a result, the normalized mix norm at $t=1$ is 2 to 3 times larger in these irregular domains than in the square or circle which is consistent with Figure \ref{fig:enstrophy norms}.  In summary, we come to an empirical conclusion: a smaller $\lambda_{1}$ implies a slower decay rate of the mix norm. 
 
\begin{figure} 
	\centering 
	\includegraphics[width=0.45\linewidth]{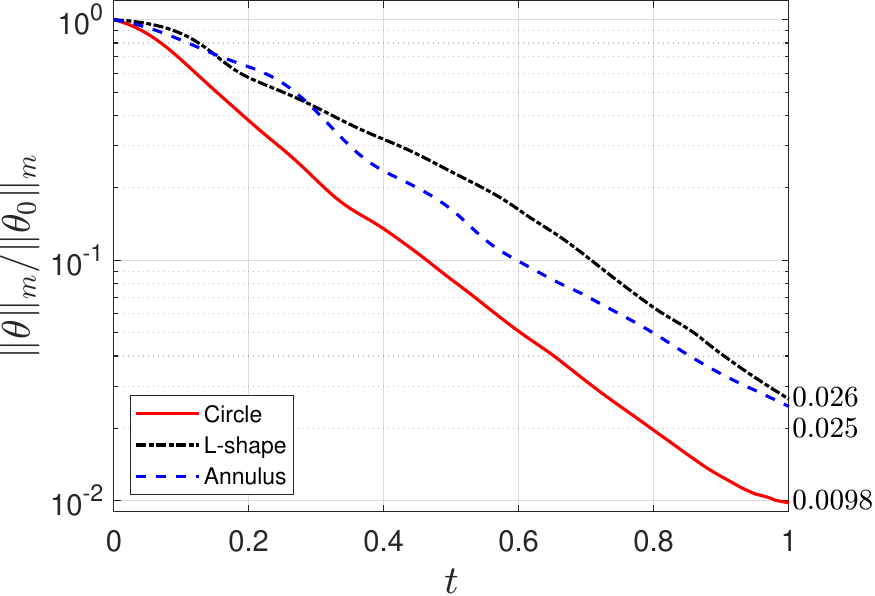}	 
	\caption{Semi-logarithmic plot of mix norm in regions with different shapes under fixed enstrophy constraint. The red solid  curve, the black dot-dashed curve and blue dashed curve represent the mix norm in the circular region, ``L'' shape region and annular region, respectively. The approximations of these curves at the final stage of the simulation are $\exp(-4.88t-0.003)$, $\exp(-3.71t+0.26)$ and $\exp(-4.02t+0.2111)$, respectively. 
        } 
	\label{fig:enstrophy norms} 
\end{figure}

\begin{table}[ht] 
\centering 
\begin{tabular}{|l|c|c|} 
\hline 
& First nonzero eigenvalue $\lambda_{1}$ &  $\frac{\Vert\theta (\mathbf{x},1 )\Vert_{m}}{\Vert\theta (\mathbf{x}, 0)\Vert_{m}}$    \\ \hline 
Square \eqref{eq:square domain}& $\pi^{2}/4\approx$2.4674  & 0.006115 \\ \hline 
Circle  \eqref{eq:circular domain}& 2.66422 & 0.009823\\ \hline 
``L'' -shape    \eqref{eq:L domain}& 1.31596 & 0.02646  \\ \hline 
Annulus \eqref{eq:donut domain}& 1.24891 & 0.02469\\ \hline 
\end{tabular} 
\caption{ The smallest eigenvalue  $\lambda_{1}$ of the Laplace equation with Neumann boundary condition on different domain. The third column lists the normalized mix norm  $\frac{\Vert\theta (\mathbf{x},1 )\Vert_{m}}{\Vert\theta (\mathbf{x}, 0)\Vert_{m}}$ at $t=1$ for the simulations presented in Figure \ref{fig:enstrophy square} and \ref{fig:enstrophy circle}.   } 
\label{tab:eigenvalue} 
\end{table}

\subsection{Dependence of the mix norm decay rate on the initial condition}

We also investigate the impact of the initial condition on the decay rate of the mix norm. Lin et al. \cite{lin2011optimal} defined a characteristic length scale for the initial distribution as
\begin{equation} 
l_{0}=\frac{\Vert\theta_{0}\Vert_{m}}{\Vert\theta_{0}\Vert_{L^{2}}}, 
\end{equation} 
and it was reported that the exponential decay rate is insensitive to the change in this parameter.  We also conduct several simulations under a fixed enstrophy constraint with initial distributions of variable length scales computed as above.  The results are summarized in Table \ref{table:l0} and similar to the  periodic case, a significant difference in the characteristic length scale of the scalar field does not entail a noticeable change in the exponential decay rate of the mix norm.
 
\begin{table}[] 
	\begin{tabular}{|c|c|c|c|c|} 
		\hline 
No.	&Initial function $\theta_{0}$& $\frac{\Vert\theta (\mathbf{x},1 )\Vert_{m}}{\Vert\theta (\mathbf{x}, 0)\Vert_{m}}$&Fitted decay rate\\ \hline 
		$1$&\multicolumn{1}{l|}{$\cos(4\pi x)+1.357\cos(6\pi y)$}&  0.0165  & -4.2017 \\ \hline 
		$2$&\multicolumn{1}{l|}{$0.566\cos(4\pi x)+\cos(6\pi y)+\cos(2\pi x)\cos(4\pi y)$} & 0.0143   &-4.4952  \\ \hline 
		\multirow{2}{*}{$3$}&\multicolumn{1}{l|}{$0.3\cos(4\pi x)+\cos(6\pi x)+0.896\cos(6\pi y)$} &\multirow{2}{*}{0.0142}&\multirow{2}{*}{-4.3123}\\  
		&\multicolumn{1}{l|}{$+2\cos(2\pi x)\cos(4\pi y)$} &&\\ \hline 
		$4$&\multicolumn{1}{l|}{$\sin(10\pi y)+0.2496(x-0.5)(x+0.5)$}&0.0167&-4.3557  \\ \hline 
		$5$&\multicolumn{1}{l|}{$\cos(\pi x)+0.3\cos(2\pi y)$} &  0.0133  & -4.2852 \\\hline 
		$6$&\multicolumn{1}{l|}{$2\cos(\pi x)+0.3\cos(2\pi y)+0.92\cos(\pi x)\cos(\pi y)$} & 0.0126   &-4.4076  \\\hline 
		$7$&\multicolumn{1}{l|}{$3\cos(\pi x)+0.7\cos(2\pi y)+\cos(\pi x)\cos(\pi y)$}&  0.0128  &-4.3620  \\\hline 
		$8$&\multicolumn{1}{l|}{$x^4+0.273y(y-0.4)$} &   0.0234 & -3.9690 \\ \hline 
	\end{tabular} 
	\caption{The functions labeled as No. 1-4  satisfy $l_{0}=0.06$. The functions labeled as No. 5-8 satisfy $l_{0}=0.3$.  Their normalized mix norms  at the end time $t=1$ are displayed in the third column. We fit the mix norm with $exp(at+b)$, where the value $a$ is displayed in the fourth column.  } 
	\label{table:l0} 
\end{table}

\subsection{Effect of boundary conditions: periodic versus no-flux}\label{Compare} 
In previous sections, we have conducted a comprehensive study for the optimal strategies \eqref{eq:optimal energy} and \eqref{eq:optimal enstrophy} in various settings from both theoretical and numerical prospectives. A pertinent question emerges from their striking resemblance to established results for a periodic boundary \cite{lunasin2012optimal,lin2011optimal}: Exactly how do different types of boundary conditions impact the optimal stirring strategy? We will see in the following that the answer depends on whether there exists a symmetry in the initial field. Without loss of generality, we focus on the two dimensional square $[-1,1]^2$.
 
\subsubsection{Even initial conditions} 

For initial scalar distribution that is even in both directions, namely,
\begin{equation}
	\theta_0(x,y)=\theta_0(-x,y), \quad \theta_0(x,y)=\theta_0(x,-y),
	\label{eq:even_init}
\end{equation}
the optimal stirring with a periodic boundary maintains the symmetry and there is no scalar flux across the edges for all times. Therefore it will be identical to the strategy \eqref{eq:optimal energy} or \eqref{eq:optimal enstrophy} depending on different velocity constraints.

Simulations for a series of initial scalar distributions satisfying (\ref{eq:even_init}) have been performed to verify this equivalence. For each $\theta_0$, the scalar field is advected by either the optimal stirring with a periodic boundary or with a no-flux one with the same set of numerical and constraint parameters. As expected, two sets of results are effectively identical without exception. For example, for the initial field  
\begin{equation}\label{eq:even initial} 
\theta_{0}=\cos(2\pi x)\cos(\pi y)+0.5\cos(2\pi y), 
\end{equation}  
 scalar evolutions are visibly identical and the mix norms produced by numerics for two types of boundary conditions only differ by a negligible relative error of $0.0012$ throughout the entire simulations conducted on a $512^2$ grid. This consistency holds for both fixed energy and fixed enstrophy constraints.

\subsubsection{  Non-even initial conditions} 
For general initial conditions without the symmetry structure (\ref{eq:even_init}), the optimal stirring with a periodic boundary does not support impermeable boundaries and would inevitably deviate from the no-flux cases. As an illustrative example, we start by simulating the distribution (\ref{eq:energy_theta}). Here the parameter for the fixed energy constraint in equation (\ref{eq:fixed energy}) is again $U=1$. Figure \ref{fig:energy noeven comparison} depicts the evolution of the scalar field advected by the optimal flow subject to the fixed energy constraint and a periodic boundary as a comparison to the no-flux results presented in Figure \ref{fig:energy_eg_square}.

\begin{figure}[ht]
	\centering 
	\includegraphics[width=1\linewidth]{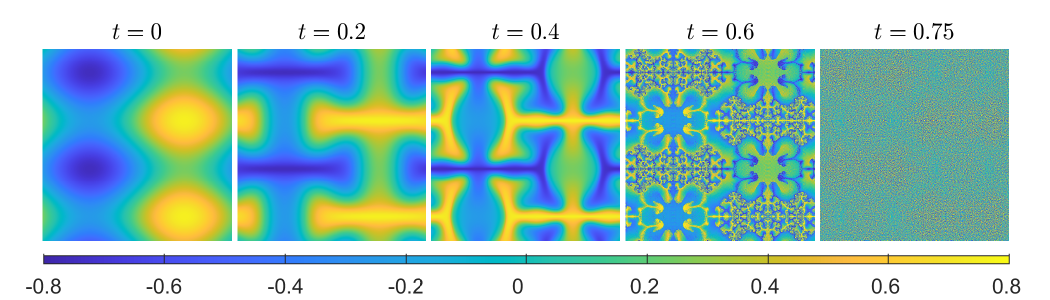} 
	\caption{Evolution of the scalar field in $[-1, 1]^2$ with periodic boundaries under fixed energy constraint \eqref{eq:fixed energy} and $U=1$, the initial condition is \eqref{eq:energy_theta}.
	} 
	\label{fig:energy noeven comparison}	 
      \end{figure} 
  
Scalar evolutions exhibit significant differences in Figures \ref{fig:energy_eg_square} and \ref{fig:energy noeven comparison} although they both start with the same initial scalar distribution. Specifically, the scalar field under no-flux boundary conditions displays less filamentation and smaller-scale structures compared to that with periodic boundaries. Furthermore, a more prominent periodic and self-similar structure can be seen in Figure \ref{fig:energy noeven comparison}. This discrepancy is attributed to the no-flux condition that limits transport across boundaries and thereby constrains the stretching and folding of the scalar field. Moreover, upon closer inspection of the scalar field at $t=0.9$ in Figure \ref{fig:energy_eg_square}, it becomes evident that the scalar field near the left and right boundaries is less well- mixed compared to the central region of the domain. By contrast, the periodic boundaries allow the scalar field to wrap around from one side of the domain to the other, resulting in a more uniform mixing degree across the entire domain as shown in Figure 8. The mix norms depicted in panel (a) of Figure 9 seem to reinforce this observation in that the curve corresponds to periodic boundary conditions exhibits a faster decay, especially in the later stages, compared to that for a no-flux boundary.

 \begin{figure}[ht] 
 \centering 
 \subfigure[]{ 
   \includegraphics[width=0.418\linewidth]{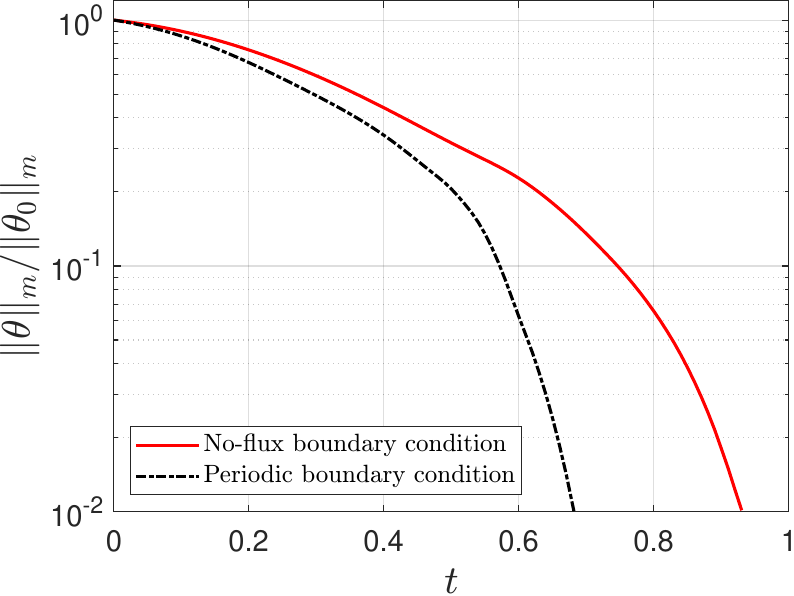} } 
 \subfigure[]{ 
   \includegraphics[width=0.46\linewidth]{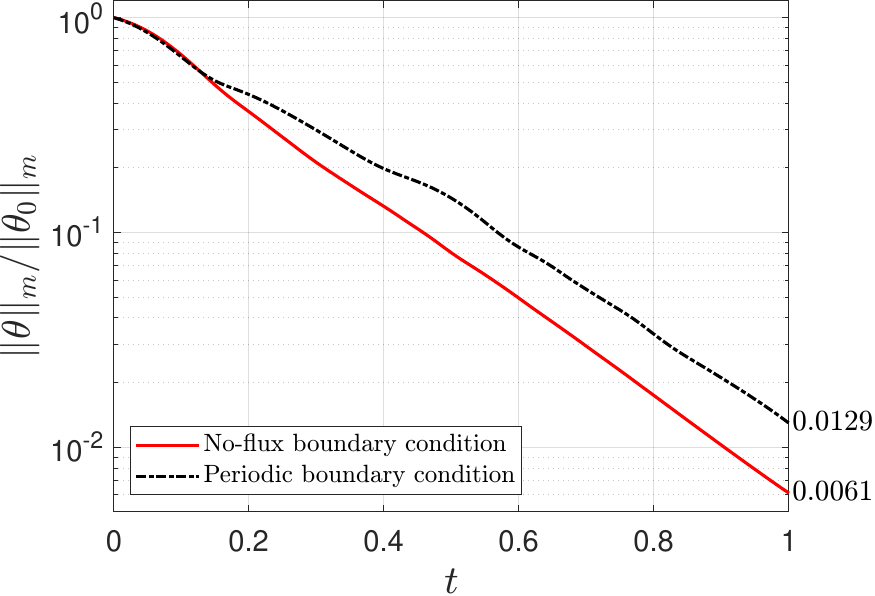} 
}  
  \hfill 
  \caption[]{ Semi-logarithmic plot of mix norms under two boundary conditions for the initial condition provided in \eqref{eq:energy_theta} with  (a)  fixed energy constraint  and (b) fixed enstrophy constraint.} 
	\label{fig:noeven norm comparison} 
\end{figure} 

After examining the previous example, one might anticipate that the no-flux boundary condition consistently reduces mixing efficiency compared to the periodic case. However, the result is quite different when considering the fixed enstrophy constraint. We conducted simulations starting with the initial condition \eqref{eq:energy_theta} under the fixed enstrophy constraint ($\tau^{-1}=15$) for both types of boundary conditions. As illustrated in panel (b) of Figure \ref{fig:noeven norm comparison}, it is evident that the mix norm decreases more rapidly under the no-flux boundary condition compared to the periodic case. 

\begin{figure} 
	\centering 
	\includegraphics[width=1\linewidth]{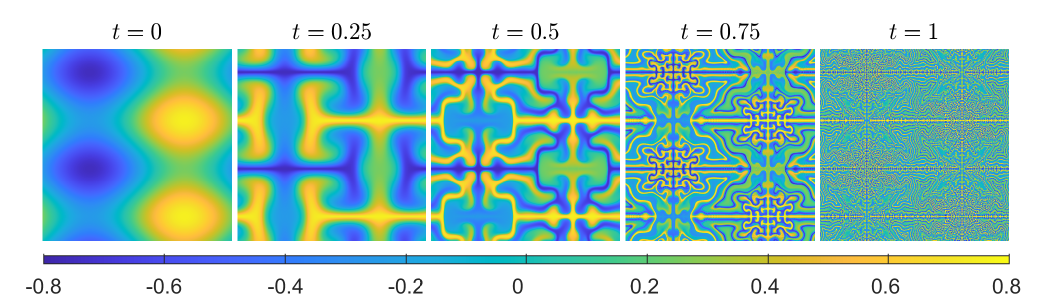} 
	\caption{Evolution of the scalar field in $[-1, 1]^2$ with periodic boundaries under fixed enstrophy constraint \eqref{eq:fixed enstrophy} and $\tau^{-1}=15$, the initial condition is \eqref{eq:energy_theta}.} 
	\label{fig:enstrophy_periodic}	 
      \end{figure}
      
Figure \ref{fig:enstrophy_periodic} illustrates the scalar evolution in the periodic case, and the no-flux situations are already presented in Figure \ref{fig:enstrophy square}. Unsurprisingly, the periodic boundary condition permits flow across the boundary, resulting in a periodic structure observed in the scalar field. In contrast, the no-flux boundary condition restricts flow across the edges, inducing circulation in the flow. This enforces the flow distribute more energy in larger scales.  In turn, it enhances scalar transport in larger scales and leads to a more well-mixed scalar field at $t=1$ compared to the periodic case. 
 
After testing a wide range of initial conditions,  we observe that periodic boundaries tend to induce a faster decay in mix norm compared to no-flux conditions under the fixed energy constraint, while the comparison is reversed for the fixed enstrophy constraint.  A more thorough justification for this result will be discussed in future work.

\section{Conclusion and future work}\label{sec:Discussion}

This paper investigates the optimal stirring strategy to maximize the instantaneous decay rate of the mix norm of a passive scalar field.  We demonstrate that the method used for periodic domains \cite{lin2011optimal} can be generalized for domains with general shapes and a no-flux boundary. The explicit expression for the optimal flow is derived under a fixed energy constraint, as provided in equation \eqref{eq:optimal energy}.  We establish that under the fixed enstrophy constraint the flow \eqref{eq:optimal enstrophy} is the optimal stirrer for a rectangular region, while it serves as a quasi-optimal strategy for other domains. Numerical simulations indicate that the mix norm of the scalar field advected by either of the optimal stirring strategies \eqref{eq:optimal energy} and \eqref{eq:optimal enstrophy} decays exponentially, indicating their mixing effectiveness in different geometry and constraint settings. 
 
Additionally we report two empirical findings based on numerical evidence.  First, the decay rate of the scalar mix norm is insensitive to the initial distribution.  Second, the domain geometry has a significant impact on the decay rate via the smallest, non-zero eigenvalue of the Laplace operator that we denote by $\lambda_{1}$. Specifically, a smaller $\lambda_{1}$ indicates a slower decay rate of the mix norm. 

Finally we explore the effect of boundary conditions on optimal stirring in a rectangular region. Theoretical and numerical results agree that when the initial condition is even, both no-flux and periodic boundary conditions yield the same optimal strategy. In contrast, for general, non-even initial distributions the optimal strategy differs for periodic and non-periodic boundaries.  In particular, simulations suggest that under a fixed energy constraint periodic boundaries result in generally faster mix norm decay and therefore a more pronounced mixture compared to no-flux boundaries.  By contrast, under a fixed enstrophy constraint no-flux boundary conditions tend to entail an optimal stirring that produces faster mixing. 

There are several possibilities for future research.  While this paper concentrates on maximizing the instantaneous decay rate of the mix norm to derive a local-in-time strategy, we will explore global optimization with negative Sobolev norms as demonstrated in \cite{Miles2018} for a reduced advection-diffusion model.  Alternatively, we look forward to further extending the methodology adopted here to other dynamical phenomena including advection-diffusion equation \cite{foures2014optimal} and advection-Cahn–Hilliard equation \cite{feng2020phase}.  Furthermore, optimal stirring with more general boundary conditions is practically intriguing such as boundaries with prescribed, non-zero scalar fluxes \cite{thiffeault2021nonuniform}.

\section{Acknowledgments} 
Sirui Zhu and Zhi Lin were supported by the National Natural Science Foundation of China (Grant Nos. 12071429 and 12090020).

\section{Appendix} 
\subsection{Norm equivalence}\label{sec:Equivalence} 
Mathematically, two norms $\Vert \cdot \Vert_{a}, \Vert \cdot \Vert_{b}$ are considered equivalent if positive constants $c_1$ and $c_2$ exist, such that for any vector or function in the vector space: 
\begin{equation} 
c_1\Vert \cdot \Vert_{b}\leq \Vert \cdot \Vert_{a}\leq c_2\Vert \cdot \Vert_{b}. 
\end{equation} 
Before proving the equivalence between the mix norm (\ref{def:mix norm}) and the $H^{-1}$ norm, we first present an alternative definition of the $H^{-1}$ norm in the following proposition: 
\begin{proposition}\label{pro:H-1 representation} 
For any function \(f(\mathbf{x}) \in H^{1}(\Omega)\), the $H^{-1}$ norm can be computed as: 
\begin{equation} 
\Vert f\Vert_{H^{-1}}=\frac{\langle f, f\rangle}{\Vert f \Vert_{H^{1}}}. 
\end{equation}	 
\end{proposition} 
\begin{proof} 
 Using the definition of $H^{-1}$ norm provided in equation \eqref{eq:H-q}, we have 
\begin{equation}\label{eq:left min} 
\Vert f\Vert_{H^{-1}}=\sup_{g\in H^{1}}\frac{\langle f, g\rangle}{\Vert g\Vert_{H^{1}}}\geq\frac{\langle f, f\rangle}{\Vert f \Vert_{H^{1}}}, 
\end{equation} 
due to the supremum over $g$.  By the Cauchy-Schwarz inequality, $\langle f, g\rangle\leq  \Vert f\Vert_{L^2}\Vert g\Vert_{L^2}$, where the equality holds if and only if $g$ is a constant multiple of $f$, i.e., \(g = kf\) for almost every $\mathbf{x}$ in \(\Omega\).    
Therefore, for all  $g\in H^{1} (\Omega)$, we have 
\begin{equation}\label{eq:right min} 
\begin{aligned} 
\frac{\langle f, g \rangle}{\Vert g \Vert_{H^1}}  \leq \frac{\Vert f\Vert_{L^2}\Vert f\Vert_{L^2}}{\Vert f \Vert_{H^1}}=\frac{\langle f, f\rangle}{\Vert f \Vert_{H^{1}}}. 
\end{aligned} 
\end{equation} 
Combining equation \eqref{eq:left min} and \eqref{eq:right min}, we have shown that 
\begin{equation} 
\begin{aligned} 
\Vert f \Vert_{H^{-1}} =  \frac{\langle f, f \rangle}{\Vert f \Vert_{H^1}}. 
\end{aligned} 
\end{equation}	This completes the proof. 
\end{proof} 
 
Subsequently, we articulate the equivalence between the mix norm and the $H^{-1}$ norm. 
\begin{proposition} 
For any function \( \theta(\mathbf{x}) \in H^{-1} (\Omega) \), the norms \( \Vert \theta \Vert_{H^{-1}} \) and \( \Vert \theta \Vert_{m} \) are equivalent.	 
\end{proposition} 
\begin{proof} 
  First, we aim to show $\Vert \theta\Vert_{H^{-1}} \leq \Vert \theta \Vert_{m}$.  Using the integration by parts and  divergence theorem, we have 
\begin{equation} 
\begin{aligned} 
 \int_{\Omega}\theta^{2}\mathrm{d}\mathbf{x}&=\int_{\Omega}\theta\Delta\varphi\mathrm{d}\mathbf{x}=\int_{\Omega}\left(\nabla\cdot(\theta\nabla\varphi)-\nabla\theta\cdot\nabla\varphi\right)\mathrm{d}\mathbf{x}\\ 
&=-\int_{\Omega}\nabla\theta\cdot\nabla\varphi\mathrm{d}\mathbf{x}+\int_{\partial\Omega}\theta\nabla\varphi\cdot\mathbf{n}\mathrm{d}\mathbf{x}=-\int_{\Omega}\nabla\theta\cdot\nabla\varphi\mathrm{d}\mathbf{x}, 
\end{aligned} 
\end{equation} 
where $\varphi=\Delta^{-1}\theta$ and the inverse Laplace operator $\Delta^{-1}$ is defined in equation \eqref{eq:laplace no flux}.  Applying Cauchy-Schwarz inequality to the  above equation yields $\Vert \theta\Vert_{L^{2}}^{2}\leq \Vert \nabla \theta\Vert_{L^{2}}\Vert \nabla \varphi \Vert_{L^{2}}=\Vert \nabla \theta\Vert_{L^{2}} \lVert \theta\rVert_{m} $. Therefore,  we have 
\begin{equation} 
\begin{aligned} 
\Vert \theta\Vert_{H^{-1}}=\frac{\langle \theta, \theta\rangle}{\Vert \theta \Vert_{H^{1}}}\leq \frac{\Vert \nabla \theta\Vert_{L^{2}} \lVert \theta\rVert_{m}}{\Vert  \theta\Vert_{H^{1}}} \leq  \lVert \theta\rVert_{m}, 
\end{aligned} 
\end{equation} 
where the first equality follows from Proposition \ref{pro:H-1 representation}, the last inequality  $\frac{\Vert \nabla \theta\Vert_{L^{2}} }{\Vert  \theta\Vert_{H^{1}}} \leq 1$ follows from the definition of the $H^{1}$ norm.

Next, we aim to show $\Vert \theta\Vert_{H^{-1}} \geq c_{1} \Vert \theta \Vert_{m}$ for a constant $c_{1}$.   From the definition of $H^{-1}$ norm, we have 
\begin{equation}\label{eq:equivalence left} 
\begin{aligned} 
\Vert \theta \Vert_{H^{-1}} = \sup_{g\in H^{1}}\frac{\langle \theta, g\rangle}{\Vert g\Vert_{H^{1}}}=\sup_{g\in H^{1}}\frac{\langle \theta, g\rangle}{\sqrt{\lVert g\rVert_{L_{2}}^{2}+\lVert \nabla g\rVert_{L_{2}}^{2} }} \geq \sup_{g\in H^{1}}\frac{\langle \theta, g\rangle}{\sqrt{(1+C^2)}\Vert\nabla g\Vert_{L^{2}} }, 
\end{aligned} 
\end{equation} 
which follows the Poincar\'e inequality which states that  \(\Vert g\Vert_{L^{2}} \leq C\Vert \nabla g\Vert_{L^2}\) for a constant $C$.  Taking $g=-\varphi =- \Delta^{-1}\theta$ in above equation leads to 
\begin{equation}\label{eq:equivalence right} 
\begin{aligned} 
\Vert \theta\Vert_{H^{-1}} \geq \frac{\langle \theta,-\varphi\rangle}{\sqrt{1+C^2}\Vert \nabla\varphi\Vert_{L^{2}}}=\frac{\langle \nabla\varphi, \nabla\varphi\rangle}{\sqrt{1+C^2} \Vert \nabla\varphi\Vert_{L^{2}}}= \frac{1}{\sqrt{1+C^2}} \Vert\theta\Vert_{m},  
\end{aligned} 
\end{equation} 
where the equality above is deduced from $\int_{\Omega}|\nabla\varphi|^{2}\mathrm{d}\mathbf{x}=\int_{\Omega}\left(\nabla(\varphi\nabla\varphi)-\varphi\Delta\varphi\right)\mathrm{d}\mathbf{x}=-\int_{\Omega}\varphi\theta\mathrm{d}\mathbf{x}$. Combining equation \eqref{eq:equivalence left} and \eqref{eq:equivalence right}, we have 
\begin{equation} 
\Vert\theta\Vert_{H^{-1}}\leq\Vert\theta\Vert_{m}\leq\sqrt{1+C^2}\Vert\theta\Vert_{H^{-1}}. 
\end{equation}	 
We establish the equivalence between the \(H^{-1}\) and \(m\) norms, which  completes  the proof.	 
\end{proof}

\subsection{Even extension for the rectangular domain} 
\label{sec:Even extension for the square domain} 
This section outlines the procedure for demonstrating that the boundary integral term in \eqref{eq:dJpuv} is zero for a rectangular domain, thereby establishing that equation \eqref{eq:optimal enstrophy} represents the optimal flow under the fixed enstrophy constraint. 
 
Without loss of generality, we focus on a unit square region $\Omega=\left\{ (x,y)| x,y \in [0,1] \right\}$. The domain after the even extension is $\Omega=\left\{ (x,y)| x,y \in [-1,1] \right\}$. The even extension of the function $f$ is defined as 
\begin{equation} 
\tilde{f} (x,y) = 
\begin{cases} 
  f (x,y) & 0\leq x\leq 1, 0\leq y\leq 1\\ 
  f(-x,y) &-1\leq x<0, 0\leq y\leq 1\\ 
  f(x,-y)&0\leq x\leq 1, -1\leq y<0\\ 
  f(-x,-y)&-1\leq x<0, -1\leq y<0\\   
\end{cases}. 
\end{equation} 
We denote $\Vert\theta\Vert_{m,\Omega}=\lVert \nabla \Delta_{\Omega}^{-1}\theta\rVert_{L^{2}}$ as the mix norm on the domain $\Omega$, where $\varphi=\Delta_{\Omega}^{-1}\theta$ solves 
\begin{equation} 
\begin{aligned} 
\Delta\varphi&=\theta \quad \text{in}~~\Omega, \quad \left. \partial_{\mathbf{n}} \varphi \right|_{\mathbf{x} \in \partial \Omega }=0. 
\end{aligned}	 
\end{equation}     
We denote $\Vert\theta\Vert_{m,\tilde{\Omega}}=\lVert \nabla \Delta_{\tilde{\Omega}}^{-1}\theta\rVert_{L^{2}}$ as the mix norm on the extended domain $\tilde{\Omega}$, where $\varphi=\Delta_{\Omega}^{-1}\theta$ solves  
\begin{equation} 
\begin{aligned} 
\Delta\varphi&=\theta \quad \text{in}~~\tilde{\Omega}, \quad \varphi \text{ is periodic}. 
\end{aligned}	 
\end{equation} 
The subsequent proposition is crucial to the proof. 
\begin{proposition} 
 Inverting Laplace operator commutes with the even extension operation: $\widetilde{\Delta_{\Omega}^{-1} \theta}= \Delta_{\tilde{\Omega}}^{-1}\tilde{\theta}$. 
\end{proposition} 
\begin{proof}	 
  The eigenfunctions $\phi_{n,m}=2\cos n(\pi x) \cos (m\pi y)$ and corresponding eigenvalues $\lambda_{n,m}= (n^{2}+m^{2})\pi^{2}$ of the Laplace operator  satisfy 
 \begin{equation} 
\begin{aligned} 
\Delta\phi_{n,m}&=-\lambda_{n,m} \phi_{n,m} \quad \text{in}~~\Omega, \quad \left. \partial_{\mathbf{n}} \phi_{n,m} \right|_{\mathbf{x} \in \partial \Omega }=0, 
\end{aligned} 
\end{equation} 
where $n,m$ are nonnegative integers not simultaneously zero.  In addition, $\lambda_{0,0}=0$, $\phi_{0,0}=1$. The scalar $\theta$ admits the following  series expansion: $\theta= \sum\limits_{n,m}^{} a_{n,m}\phi_{n,m}$, where $a_{n,m}=\left\langle \theta, \phi_{n,m} \right\rangle$. Since the scalar is assumed to have zero mean, $a_{0,0}=0$. Therefore, we have  
\begin{equation} 
\varphi=\Delta_{\Omega}^{-1}\theta=\sum\limits_{(n,m)\neq (0,0)}^{} \frac{-a_{n,m} \phi_{n,m}}{\lambda_{n,m}}. 
\end{equation} 
 Due to the even and periodic nature of $\phi_{n,m}$ on the extended domain $\tilde{\Omega}$, the extension $\tilde{\varphi}$ shares the exact expression with $\varphi$: 
 \begin{equation} 
\tilde{\varphi}=\widetilde{\Delta_{\Omega}^{-1}\theta}=\sum\limits_{(n,m)\neq (0,0)}^{} \frac{-a_{n,m} \phi_{n,m}}{\lambda_{n,m}}. 
 \end{equation} 
Similarly, $\tilde{\theta}= \sum\limits_{(n,m)\neq(0,0)}^{} a_{n,m}\phi_{n,m}$ and $\Delta_{\tilde{\Omega}}^{-1}\tilde{\theta}=\sum\limits_{(n,m)\neq (0,0)}^{} -\frac{a_{n,m} \phi_{n,m}}{\lambda_{n,m}}$. Therefore, $\widetilde{\Delta_{\Omega}^{-1} \theta}= \Delta_{\tilde{\Omega}}^{-1}\tilde{\theta}$.    This completes the proof.  
\end{proof} 
 
Using the same approach, we can demonstrate that differentiation also commutes with the extension operation. Consequently, we can establish the relationship $||\theta||_{m,\Omega}^{2}=||\tilde{\theta}||_{m,\tilde{\Omega}}^{2}$. 
 
With above conclusions, the Lagrange multiplier formalism \eqref{eq:Lagrange multiplier formalism} is equivalent to the following functional defined in the extended domain: 
\begin{equation} 
\begin{aligned} 
&J_{en}(\tilde{\mathbf{u}})=-2\int_{\tilde{\Omega}}\tilde{\mathbf{u}}\cdot \mathscr{P}(\tilde{\theta}\nabla\tilde{\varphi})d\mathbf{x}+\lambda \left(  \int_{\tilde{\Omega}}|\nabla\tilde{\mathbf{u}}|^{2}d\mathbf{x}-\frac{\left| \tilde{\Omega} \right|}{\tau^{2}} \right) \quad \text{on}~\tilde{\Omega}.\\ 
\end{aligned} 
\end{equation} 
Its functional derivative is  
\begin{equation} 
\begin{aligned} 
dJ_{en} (\tilde{\mathbf{u}};\tilde{\mathbf{v}})=-2\int_{\tilde{\Omega}}\tilde{\mathbf{v}}\cdot \left( \Delta \tilde{\mathscr{P}}\Delta_{\tilde{\Omega}}^{-1}(\tilde{\theta}\nabla\tilde{\varphi})+\lambda \Delta \tilde{\mathbf{u}}  \right)\mathrm{d}\mathbf{x}+2\lambda \int_{\partial\tilde{\Omega}} \sum\limits_{i,j=1}^{d}\tilde{v}_{i}n_{j}\partial_{x_{j}}\tilde{u}_{i}\mathrm{d}\mathbf{x}, 
\end{aligned} 
\end{equation} 
where $\tilde{\mathscr{P}} (\mathbf{v})= \mathbf{v}- \nabla \Delta_{\tilde{\Omega}}^{-1} (\mathbf{v})$. The periodicity ensures that the integral along the opposite edge cancels out. Thus, the boundary integral term in the above equation is zero. The optimal stirring field on the extend domain is  
\begin{equation} 
\tilde{\mathbf{u}}_{en}=\frac{1}{\tau}\frac{-\tilde{\mathscr{P}}\Delta_{\tilde{\Omega}}^{-1}(\tilde{\theta}\nabla\tilde{\varphi})}{\langle|\nabla\tilde{\mathscr{P}}\Delta_{\tilde{\Omega}}^{-1}(\tilde{\theta}\nabla\tilde{\varphi})|^2\rangle^{1/2}}, 
\end{equation} 
which is the extension of the flow provided in \eqref{eq:optimal enstrophy}. Therefore, equation \eqref{eq:optimal enstrophy} is the optimal flow under fixed enstrophy constraint in a square domain.  
 
The key to the above argument is the commutativity of the inverse Laplace operator and the even extension operation. Therefore, it is straightforward to show that equation \eqref{eq:optimal enstrophy} is the optimal flow for an arbitrary rectangular domain. However, this commutativity may not be valid for more general domains, and equation \eqref{eq:optimal enstrophy} is a quasi-optimal flow, not the optimal flow in these domains. 
 
\subsection{Numerical methods\label{sec:numerical_methods}} 
This section documents the details of numerical simulations that we use to simulate scalar evolution and to compute the optimal stirring in different settings. 
\subsubsection{Spectral method} 
\label{sec:Spectral method} 
The (pseudo) Fourier spectral method is one of the efficient methods for the simulation in the domains with regular geometries \cite{trefethen2000spectral,boyd2001chebyshev,ding2021enhanced,ding2022determinism}. Therefore, this method is used for the simulation for the square region in this paper. Within this methodology, all functions are approximated using Fourier series, transforming the derivative operation into a multiplication in the spectral space. The computation of nonlinear terms is streamlined through Fast Fourier Transform (FFT)-based convolution, and for more in-depth insights, readers are directed to the referenced textbooks \cite{trefethen2000spectral,boyd2001chebyshev}.

We employ the explicit 4th-order Runge-Kutta method as our time-marching scheme. To enhance efficiency, instead of computing the optimal flow and advecting the scalar field at each time step, we exclusively calculate the optimal flow at time instances that differ by a larger time step size, denoted as $\Delta t$. Upon obtaining the optimal flow, we employ it to advect the scalar field over a time interval of length $\Delta t$ using the Runge-Kutta method with a smaller time step size that satisfies the Courant–Friedrichs–Lewy condition.
 
In the dealiasing process at each time step, an all-or-nothing filter is incorporated into the spectrum. In simulations adhering to the fixed enstrophy constraint, the two-thirds rule is applied, nullifying the upper one-third of the resolved spectrum. However, in simulations with the fixed energy constraint, the velocity expressed in equation \eqref{eq:optimal energy} yields large velocity gradients during the simulation. To ensure numerical stability, the upper half of the spectrum is zeroed at each time step.

Although the Fourier spectral method  is particularly effective in periodic domains, we encounter no-flux boundary conditions in our study. Taking the square domain defined in \eqref{eq:square domain} as an example, we address this challenge by implementing an even extension to establish periodic conditions on the extended domain $\left\{(x,y) | x,y\in [-1, 3]\right\}$. This extension guarantees no-flux boundary conditions on the original domain $\left\{(x,y) | x,y\in [-1, 1]\right\}$. The original domain comprises $513\times513$ grid points, while the extended domain encompasses $1024\times1024$ grid points.

\subsubsection{Finite element method} 
\label{sec:finite element method} 
The finite element method is employed to simulate scalar evolution in non-rectangular regions, utilizing the FreeFEM++ software (\cite{hecht2012new}).  In the simulation, $P2$ elements are employed, which is the set of bivariate polynomials of degrees $\leq 2$ over the triangles. 
 
At each time step, we begin by computing the optimal flow as defined in equations \eqref{eq:optimal energy} or \eqref{eq:optimal enstrophy}. The computation of the optimal flow requires solving equations \eqref{eq:laplace no flux} and \eqref{eq:Leray-Helmholtz projection p}, both of which are Poisson equations with Neumann boundary conditions. These equations have infinite solutions, differing by a constant. To ensure a unique solution, we impose an additional condition: the solution must have a mean of zero. The imposition of the mean zero condition, coupled with the weak form of the Poisson equation (\cite{alberty1999remarks}), results in a linear system solved by the built-in solver in FreeFEM++. A detailed overview of the numerical scheme can be found in \cite{hecht2005freefem++}. 
 
Having obtained the optimal flow, we then advance the solution of equation \eqref{eq:advection equation} using the built-in advection operator in FreeFEM++.

\end{document}